\documentclass[reqno,oneside,12pt]{amsart}
\usepackage{
  amsmath,
  amssymb,
  amsthm,
  geometry,
  paralist,
  thmtools,
  tikz,
  tikz-cd,
  verbatim,
  relsize, %mathlarger
  bbm, %blackboard numbers
  dsfont,
  mathtools
}

\usepackage[T1]{fontenc}

\usepackage{caption}
\usepackage{subcaption}
\usetikzlibrary{angles}

\renewcommand\labelenumi{(\roman{enumi})}
\renewcommand\theenumi\labelenumi

\RequirePackage{ifpdf}
\ifpdf
  \IfFileExists{pdfsync.sty}{\RequirePackage{pdfsync}}{}
  \RequirePackage[pdftex,
   colorlinks = true,
   urlcolor = myblue, % \href{...}{...} external (URL)
   citecolor = mygreen, % \cite{...}
   linkcolor = myred, % \ref{...} and \pageref{...}
   bookmarksopen=true,
   unicode, psdextra]{hyperref}
\else
  \RequirePackage[hypertex, unicode, psdextra]{hyperref}
\fi
\hypersetup{
    colorlinks=true,
    linkcolor=blue,
    filecolor=magenta,      
    urlcolor=cyan,
    citecolor = blue
} 

\RequirePackage{ae, aecompl, aeguill} % to have pretty pdf

\usepackage{quiver}

\usepackage{cleveref} %cref 

\usepackage{adjustbox}

\tikzcdset{
	arrows = crossing over
}

\newtheorem{thmx}{Theorem}

\usepackage{xurl}

\usepackage{dynkin-diagrams}

\usepackage[boxsize=.5em]{ytableau}

\usetikzlibrary{arrows.meta} % not arrows, but arrows.meta

\usepackage[textsize=scriptsize]{todonotes}
\presetkeys{todonotes}{inline}{}

\makeatletter
\providecommand\@dotsep{5}
\makeatother

\newcommand{\N}{\mathbb N}

\newcommand{\R}{\mathbb R}

\newcommand{\Z}{\mathbb Z}

\newcommand{\e}{\text{even}}

\newcommand{\<}{\langle}
\renewcommand{\>}{\rangle}

\DeclareMathOperator{\id}{id}

\newcommand{\indexset}{\mathcal{S}}

\DeclareMathOperator{\End}{End}

\DeclareMathOperator{\FP}{FP}
\DeclareMathOperator{\Span}{Span}
\DeclareMathOperator{\Fib}{Fib}
\DeclareMathOperator{\Mat}{Mat}
\DeclareMathOperator{\Conv}

\DeclarePairedDelimiter\floor{\lfloor}{\rfloor}

\declaretheorem[numberwithin=section]{Theorem}
\declaretheorem[sibling=Theorem]{Lemma}
\declaretheorem[sibling=Theorem]{Corollary}

\declaretheorem[sibling=Theorem]{Proposition}

\declaretheorem[sibling=Theorem, style=remark]{Remark}
\declaretheorem[sibling=Theorem, style=definition]{Definition}
\declaretheorem[sibling=Theorem, style=definition]{Example}

\newtheorem*{theorem*}{Theorem}

\crefname{lemma}{Lemma}{Lemma}
  \crefname{corollary}{Corollary}{Corollary}
  \crefname{theorem}{Theorem}{Theorem}
  \crefname{definition}{Definition}{Definition}
   \crefname{proposition}{Proposition}{Proposition}
\crefname{question}{Question}{Question}
 \crefname{section}{Section}{Section} 
   \crefname{construction}{Construction}{Construction}
   \crefname{generalization}{Generalization}{Generalization}
  \crefname{construction}{Construction}{Construction}
  \crefname{notation}{Notation}{Notation}
   \crefname{example}{Example}{Example}
  \crefname{remark}{Remark}{Remark}
  \crefname{fact}{Fact}{Fact}
  \crefname{conjecture}{Conjecture}{Conjecture}
  \crefname{motivation}{Motivation}{Motivation}  
  \crefname{figure}{Figure}{Figure}  
  \crefname{assumption}{Assumption}{Assumption}

\author[]{Max Boyle}
\address{M.\ Boyle: School of Mathematics and Statistics, University of Sydney, Australia}
\email{mboy9010@uni.sydney.edu.au}

\author[]{Edmund Heng}
\address{E.\ Heng: School of Mathematics and Statistics, University of Sydney, Australia}
\email{edmund.heng@sydney.edu.au}

\title[Coxeter planes and Verlinde rings]{Coxeter planes as fixed points of Verlinde fusion rings}

\begin{document}

\begin{abstract}
	For the Coxeter groups of ADE type, we provide a construction of their Coxeter planes as fixed points of actions of hypergroups associated to Verlinde fusion rings. This builds upon the well-known ADE classification of $\Z_+$-modules over these fusion rings.
\end{abstract}

\maketitle          % creates the title page

\section{Introduction} 
One of the important invariants associated to an irreducible finite Coxeter group $W$ is its Coxeter number $h$, which can be defined as the order of any Coxeter element (the product of all generators in some particular order).
Through the faithful geometric representation $V$ of $W$, there is a (real) two-dimensional plane $P \subseteq V$ on which a distinguished Coxeter element acts as a rotation of order $h$; this plane $P$ is known as the \emph{Coxeter plane} (see \cref{sec:Cox-plane}).
The study of the Coxeter plane was first formalised by Steinberg in 1959 \cite{MR106428}.
Since then, there are several methods to construct the Coxeter plane, e.g.\ via Springer theory \cite{MR354894}.

For the Coxeter groups corresponding to type $A_3$ and type $D_4$, the Coxeter planes can also be constructed via a ``folding'' method.
Namely, the $A_3$ (resp.\ $D_4$) Coxeter graph carries a $\Z/2\Z$ (resp.\ $\Z/3\Z$) symmetry.
This induces an action on $V$, and it turns out that the Coxeter plane agrees with the subspace of fixed points of this action.
However, there are no other ``foldings'' that one can do for other types -- at least not using groups.

In this paper, we show that one can still ``fold'' to obtain Coxeter planes for all ADE types, as long as we are willing to generalise from groups to \emph{hypergroups}.
Roughly speaking, hypergroups are a generalisation of groups where elements multiply into a probabilistic spread of elements (instead of a unique element); the precise definition is in \cref{def:HG}.
Our main theorem can be stated as follows.
\begin{thmx}[=\cref{thm:main-theorem}]\label{thm:0}
    Let $\Gamma$ be an ADE Dynkin diagram (see \cref{fig:CoxDyn}) and let $V_\Gamma$ be the associated geometric representation of the Coxeter group $W$. There exists a hypergroup $\mathcal{H}_\Gamma$ acting on $V_\Gamma$ such that the subspace of fixed points of $\mathcal{H}_\Gamma$ is equal to the Coxeter plane.
\end{thmx}

One way that hypergroups arise is through fusion rings (see \cref{prop:Fusion-rings-are-HG}), and our approach is very much motivated by this.
On one hand, fusion rings (and fusion categories) have recently been used to apply more general types of ``folding'' that go beyond the crystallographic types \cite{MR4769244,EH_fusionquiver,HP_fusionarrangement,MR4806405,MR4927767}.
On the other hand, there is a well-known classification of the $\Z_+$-modules of the Verlinde fusion rings (associated to $\mathfrak{sl}_2$ at some level), which also carries an ADE classification \cite{MR1333750,MR1976459}.
Our construction of the hypergroup action in \Cref{thm:0} builds directly upon this classification.
Namely, given an ADE Dynkin diagram $\Gamma$, there is an irreducible $\Z_+$-module $M_\Gamma$ over an appropriate Verlinde fusion ring $R_{h-1}$.
The Verlinde fusion ring has an even part fusion subring $R_{h-1}^\e \subseteq R_{h-1}$, and the action of $R_{h-1}^\e$ on $M_\Gamma$ induces an action of the hypergroup of $R_{h-1}^\e$ acting on $M_\Gamma \otimes_\Z \R = V_\Gamma$.
It is the fixed points of this action that give exactly the Coxeter plane.

A brief outline of the paper is as follows.
In \cref{sec:Coxgroupplane}, we recall the relevant definitions from Coxeter theory and recall one of the methods to construct the Coxeter plane (following \cite{casselman_coxeter_elements}).
\cref{sec:fusionhypergroup} contains the definitions of fusion rings and hypergroups. We also introduce in this section the notion of an action of a hypergroup (on real vector spaces) and its corresponding notion of fixed points.
In \cref{sec:verlindeclassification}, we recall the definition of Verlinde fusion rings and the ADE(T) classification of their $\Z_+$-modules.
Finally we prove our main theorem in \cref{sec:mainthm}.

\subsection*{Acknowledgements}
The authors would like to thank Devon Stockall for introducing them to the notion of hypergroups, which ultimately led to the idea of this paper. We also thank Andrew Riesen for useful discussions on hypergroups.
This work is part of the Honours thesis of M.B.\ \cite{Boyle_hons}.
E.H.\ is supported by the Australian Research Council grant DP230100654.

\section{Preliminaries on Coxeter groups} \label{sec:Coxgroupplane}
\subsection{Coxeter groups and Coxeter elements}
%We collect preliminaries on Coxeter groups in this section and we refer the reader to \cite{MR2133266} for more details.
\begin{Definition}
    Let $\indexset = \{1,2,\ldots,n\}$ be a finite set. A \textbf{Coxeter matrix} is a matrix $m: \indexset\times \indexset\to \{1,2,\,\dots\}\cup\{\infty\}$ such that for all $i,\,j\in \indexset$
    \[m(i,j)=m(j,i)\qquad \text{and}\qquad m(i,j)=1\iff i=j\]
\end{Definition}
Corresponding to a Coxeter matrix $m$ is a Coxeter graph $\Gamma$ defined as follows.
\begin{Definition}
    The \textbf{Coxeter graph} $\Gamma \coloneqq \Gamma(m)$ of a Coxeter matrix $m$ is a simplicial graph with vertex set $\Gamma_0 \coloneq \indexset$ and edges $\{i,\,j\} \in \Gamma_1$ whenever $m(i,j)\geq 3$. When $m(i,j)\geq 4$, we also label the edge $\{i,j\}$ with $m(i,j)$. 
\end{Definition}
%we will use $\Gamma_0 = \{v_1,\ldots,v_n\}$ to denote the set of vertices; and vice versa.

\iffalse
\begin{Example}
    Below is a Coxeter matrix and its corresponding Coxeter graph.
    
\begin{center}
\begin{minipage}{0.4\textwidth}
\centering
$\begin{pmatrix}
        1&5&3&\infty\\
        5&1&2&2\\
        3&2&1&3\\
        \infty&2&3&1
    \end{pmatrix}$
\end{minipage}%
\begin{minipage}{0.1\textwidth}
\centering
$\longleftrightarrow$
\end{minipage}%
\begin{minipage}{0.4\textwidth}
\centering
\begin{tikzpicture}[
  dynkin/.style={
    line width=0.8pt,
    draw=black,
  },
  node/.style={
    circle,
    draw=black,
    fill=black,
    inner sep=0pt,
    minimum size=3pt,
  }
]
  % Nodes
  \node[node, label=left:$v_1$] (a1) at (0,0) {};
  \node[node, label=right:$v_2$] (a2) at (2,1) {};
  \node[node, label=right:$v_3$] (a3) at (2,0) {};
  \node[node, label=right:$v_4$] (a4) at (2,-1) {};
  % Edges
  \draw[dynkin] (a1) -- node[midway, above left=-2pt]{$5$} (a2);
  \draw[dynkin] (a1) --  (a3);
  \draw[dynkin] (a1) -- node[midway, below left=-2pt]{$\infty$} (a4);
  \draw[dynkin] (a3) -- (a4);
\end{tikzpicture}
\end{minipage}
\end{center}
\end{Example}
\fi

\begin{Definition}
    Given a Coxeter matrix (or equivalently a Coxeter graph), the \textbf{Coxeter group} $W$ is defined by the presentation:
    \[
    W \coloneqq \langle s_i, \ i \in \indexset \mid (s_is_j)^{m(i,j)}=e\quad\forall i,j\in \mathcal{S}\text{ such that }m(i,j)<\infty \rangle.
    \]
    Then $S \coloneqq \{s_i \mid i \in \indexset\}$ (in bijection with $\indexset$) is a set of generators for $W$, and we call $(W,S)$ a \textbf{Coxeter system} of rank $n = |S|$.
    A Coxeter system (and its Coxeter group) is said to be \textbf{irreducible} if its Coxeter graph is connected.
\end{Definition}
Note that since $m(i,i)=1$, we have $s_i^2=e$ for all $s_i\in S$, so the set of relations is equivalent to 
\[s_i^2=e, \qquad \underbrace{s_is_js_i\cdots}_{m(i,j)}=\underbrace{s_js_is_j\cdots}_{m(i,j)}.\]
In particular, $s_is_j=s_js_i$ if $m(i,j)=2$, so two generators commute if and only if the associated vertices do not have an edge between them on the Coxeter graph.

The following theorem of Coxeter classifies all finite irreducible Coxeter groups.
\begin{figure}
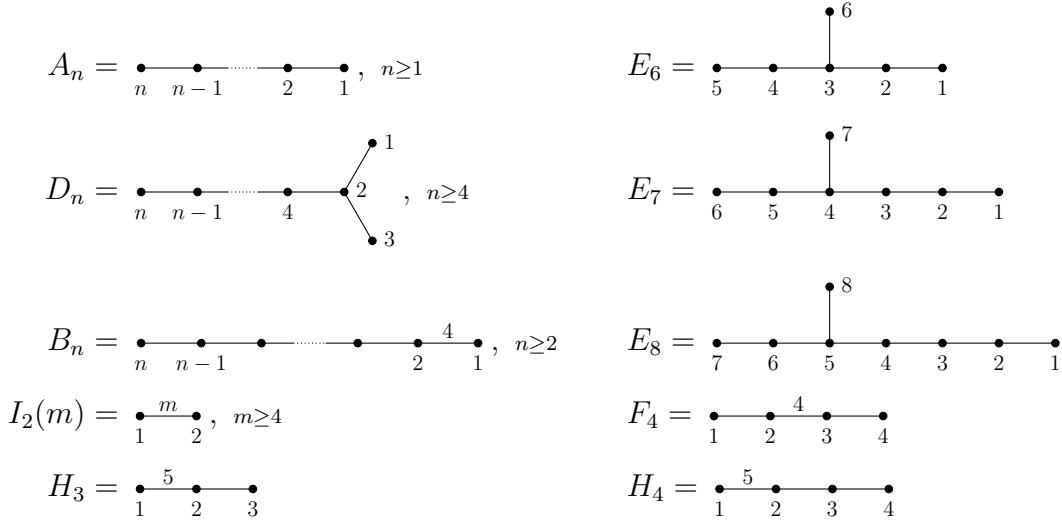

    \centering
    \begin{align*}
    A_n &= \dynkin[edge length=.75cm,labels={n,n-1,2,1},Coxeter]A{}, \ {\scriptstyle n\geq 1} \quad
    &&E_6 = \dynkin [edge length=.75cm,labels={5,6,4,3,2,1},Coxeter]E6
    \\
    D_n &= \dynkin [edge length=.75cm,labels={n,n-1,4,2,1,3},label directions={,,,right,right,right},Coxeter]D{}, \ {\scriptstyle n\geq 4} \quad
    &&E_7 = \dynkin [edge length=.75cm,labels={6,7,5,4,3,2,1},Coxeter]E7
    \\
    B_n &= \dynkin [edge length=.8cm,labels={n,n-1,,,2,1},Coxeter]B{}, \ {\scriptstyle n\geq 2} \quad
    &&E_8 = \dynkin [edge length=.75cm,labels={7,8,6,5,4,3,2,1},Coxeter]E8
    \\
    I_2(m) &= \dynkin [edge length=.75cm,labels={1,2},Coxeter,gonality=m]I{}, \ {\scriptstyle m\geq 4} \quad
    &&F_4 = \dynkin [edge length=.75cm,labels={1,2,3,4},Coxeter]F4
    \\
    H_3 &= \dynkin [edge length=.75cm,labels={1,2,3},Coxeter]H3 \quad
    &&H_4 = \dynkin [edge length=.75cm,labels={1,2,3,4},Coxeter]H4
    \end{align*}
    \caption{Coxeter--Dynkin diagrams classifying the irreducible finite Coxeter groups. The ADE Dynkin diagrams consist of the $A_n$ and $D_n$ families, and the three exceptional types $E_6, E_7$ and $E_8$.}
    \label{fig:CoxDyn}
\end{figure}
\begin{Theorem}[\protect{\cite{MR1581693}}]
    The finite irreducible Coxeter groups are classified by those whose Coxeter graphs are one of the Coxeter--Dynkin diagrams in \cref{fig:CoxDyn}.
\end{Theorem}
\iffalse
\begin{Example}
    For type $A_n$, we have that the associated Coxeter group $W(A_n)$ is isomorphic to $S_{n+1}$, the symmetric group on $n+1$ elements. Under this isomorphism, generators $s_i\in S$, $1\leq i\leq n$ correspond to two-cycles $(i,i+1)\in S_{n+1}$. As such, each generator is order $2$ and the product of adjacent generators $s_is_{i+1}$ is the three-cycle $(i,i+1)(i+1,i+2)=(i,i+1,i+2)$, which is order $3$. If two generators are not adjacent on the Coxeter graph, they will correspond to disjoint cycles, so they will commute.
\end{Example}
\begin{Example}
    $W(I_2(m))\cong D_{2m}$, the group of symmetries of a regular $m$-gon. Here the generators correspond to two reflections that differ by a rotation of order $m$.
\end{Example}
\fi
\begin{Definition}
    A \textbf{Coxeter element} is an element $s_{k_1}s_{k_2}\cdots s_{k_n}\in W$, where $s_{k_i}\in S$ for $1\leq i\leq n$ and $s_{k_i}\neq s_{k_{j}}$ for $i\neq j$. In other words, it is a product of generators such that every generator occurs exactly once.
\end{Definition}
\begin{Proposition}[\protect{See e.g.\ \cite[Theorem 1.5]{casselman_coxeter_elements}}] \label{prop:Coxelmconjugate}
    Let $W$ be an irreducible finite Coxeter group. Then all of its Coxeter elements are conjugate to one another.
\end{Proposition}
\begin{Definition}
    Let $W$ be an irreducible finite Coxeter group. By \cref{prop:Coxelmconjugate}, all Coxeter elements of $W$ have the same order, which we denote by $h$, and we call $h$ the \textbf{Coxeter number} of $(W,S)$.
\end{Definition}

Notice that every Coxeter graph $\Gamma$ associated to an irreducible finite Coxeter group is a (finite, connected) tree; see \cref{fig:CoxDyn}.
Recall that any tree $\Gamma$ admits an ordering on its set of vertices $\Gamma_0 = \{1,2, \dots, n\}$ such that for each $1 \leq k \leq n$, the set $\{1, \dots, k\}$ forms a connected subtree and $k$ is a leaf of this subtree. 
Such an ordering of the vertices for $\Gamma$ is provided in \cref{fig:CoxDyn}.
Moreover, if $d_\Gamma(i,j)$ denotes the graph distance between the vertices $i,j \in \Gamma_0$, then $\Gamma$ is bipartite with respect to the bifurcation $\Gamma_0 = \Gamma_0^+ \sqcup \Gamma_0^-$ defined by:
\begin{equation} \label{eq:bifurcation}
\begin{split}
\Gamma_0^+ &\coloneqq \{ i \in \Gamma_0 \mid d_\Gamma(i,1) \text{ is even} \}, \\
\Gamma_0^- &\coloneqq \{ i \in \Gamma_0 \mid d_\Gamma(i,1) \text{ is odd}\}.
\end{split}
\end{equation}

\begin{Definition}\label{def:distinguishCoxelm}
    Let $\Gamma$ be a Coxeter graph whose associated Coxeter group is irreducible and finite.
    Fix $\{1,2,\ldots,n\}$ to be an ordering on its set of vertices $\Gamma_0$ as above (e.g.\ in \cref{fig:CoxDyn}) with bifurcation $\Gamma_0 = \Gamma_0^+ \sqcup \Gamma_0^-$ defined in \cref{eq:bifurcation}.
    The Coxeter element of $W$ given by
    \begin{equation}
        \gamma \coloneqq \gamma^+\gamma^-,
    \end{equation}
    where
    \[
    \gamma^+ = \prod_{k \in \Gamma_0^+} s_k, \qquad 
    \gamma^- = \prod_{k' \in \Gamma_0^-} s_{k'}
    \]
    is called the \textbf{distinguished Coxeter element} (with respect to the bipartition induced from the ordering).
\end{Definition}
Note that the elements $s_k$ within $\gamma^+$ (resp.\ $s_{k'}$ within $\gamma^-$) pairwise commute due to the bipartite property, hence the order within the product is irrelevant.

\subsection{Geometric representations and Coxeter planes}
\label{sec:Cox-plane}
Fix $(W,S)$ to be a Coxeter system of rank $n$ corresponding to a Coxeter graph $\Gamma$ with vertex set $\Gamma_0 = \indexset = \{1,2,\ldots,n\}$. We construct a faithful geometric representation of $W$ as follows.

Consider the vector space $V_\Gamma \coloneqq \Span_\R\{\alpha_i \mid i \in \indexset\}$ and equip it with a bilinear form $\langle \ \cdot \mid \cdot \ \rangle: V_\Gamma \times V_\Gamma \to \R$ defined on the basis elements by
\begin{equation}\label{eq:bilinear-form}
    \langle \,\alpha_i\mid\alpha_{j}\,\rangle=-2\cos\left( \frac{\pi}{m(i,j)}\right),
\end{equation}
where it is understood that $2\cos\left( \frac{\pi}{\infty} \right) \coloneq 2$.

\begin{Definition}\label{def:symmetricCartan}
    The assignment for each $s_i \in S$ and $v \in V_\Gamma$ given by
    \[
    s_i(v) \coloneqq v - \langle \ \alpha_i \mid v \ \rangle\alpha_i
    \]
    defines an action of $W$ on $V_\Gamma$, which is moreover faithful \cite[Theorem 4.2.2 \& 4.2.7]{MR2133266}.
    This representation of $W$ is called the \textbf{standard geometric representation}.
    The matrix $\mathfrak{A}=(a_{ij})_{1\leq i,j\leq n}$ defined by 
    \begin{equation}
    a_{ij}\coloneq\langle\,\alpha_i\mid \alpha_j\,\rangle
    \end{equation}
    is called the (standard) \textbf{Cartan Matrix} of the Coxeter system $(W,S)$.
\end{Definition}
\begin{Remark}
    Note that in \cite{casselman_coxeter_elements}, the author defined the entries of $\mathfrak{A}$ by $a_{ij}=\langle \alpha_i\mid \alpha_j\rangle/2$.
\end{Remark}

Now let $W$ be an irreducible finite Coxeter group and fix an ordering $\{1,\ldots,n\} = \indexset = \Gamma_0$ so that the bifurcation of $\Gamma_0 = \Gamma_0^+ \sqcup \Gamma_0^-$ defined by \cref{eq:bifurcation} gives a bipartite graph (e.g.\ in \cref{fig:CoxDyn}).
The following result allows us to construct the Coxeter plane.
Note that we shall view the matrix $2I_n-\mathfrak A$ as acting on $V_\Gamma$, so that its eigenvectors are naturally in $V_\Gamma$.
\begin{Theorem}[\protect{\cite[Theorem 3.11]{casselman_coxeter_elements}}]\label{thm:Coxeter-plane-span}
   Let $(W,S)$ be an irreducible Coxeter system with $W$ finite.
   The matrix $2I_n-\mathfrak A$ carries two simple eigenvalues $\pm 2\cos \frac\pi h$ with corresponding eigenvectors $u^+$ and $u^-$ respectively. Under the standard geometric representation $V_\Gamma$, the distinguished Coxeter element $\gamma$ (\cref{def:distinguishCoxelm}) acts on the plane $\Span_\R\{u^+, u^-\} \subseteq V$ by a rotation of order $h$.
\end{Theorem}

\begin{Definition}\label{def:Coxeterplane}
    The plane spanned by $u^+$ and $u^-$ as in \cref{thm:Coxeter-plane-span} is called the \textbf{Coxeter plane} of $W$ (with respect to the distinguished Coxeter element $\gamma = \gamma^+\gamma^-$).
\end{Definition}
\begin{Remark}
    The subgroup $W'$ of $W$ generated by $\gamma^+$ and $\gamma^-$ is isomorphic to the Coxeter group of type $I_2(h)$.
    Moreover, we get that the action of $W'$ restricted to the Coxeter plane is exactly the standard geometric representation of $W'$.
\end{Remark}

\section{Fusion Rings and hypergroup actions} \label{sec:fusionhypergroup}
\subsection{Fusion rings and \texorpdfstring{$\Z_+$}{non-negative Z}-modules}
Let $\Z_+$ denote the semi-ring of non-negative integers.
In this subsection, we recall the definition of fusion rings and its $\Z_+$-modules following \cite[\S 3]{MR3242743}.
\begin{Definition}\label{def:fusionring}
A \textbf{$\Z_+$-ring} $R$ is a (unital, associative) ring equipped with a $\Z$-basis $B=\{b_i\}_{i\in I}$ (so it is free as a $\Z$-module) with multiplication satisfying:
\[
b_ib_j=\sum_{k\in I}c_{ij}^kb_k, \quad c_{ij}^k\in \Z_+.
\]
A $\Z_+$-ring is \textbf{unital} if its basis $B$ contains the unit element $1$; in this case, we will typically use $b_0 \in B$ to denote the unit element.

A unital $\Z_+$-ring $R$ with basis $B=\{b_i\}_{i\in I}$ is \textbf{based} if there exists an involution on the basis set, $b_i\mapsto b_i^* \coloneq b_{i^*}$ extending to an anti-automorphism on $R$ such that 
\[c_{ij}^0=\begin{cases}
    1, &\text{if $b_i=b_j^*$};\\
    0, &\text{if $b_i\neq b_j^*.$}
\end{cases}\]
A \textbf{fusion ring} is a unital based ring of finite rank (i.e. $|B|=n<\infty$).\\
A subring of a fusion ring is a \textbf{fusion subring} if it is a fusion ring with basis given by a subset of the basis of $R$.
\end{Definition}
\begin{Example} \label{eg:ZG-fusion-ring}
    Let $G$ be a finite group. Then its group ring $\Z[G]$ is a fusion ring with basis $B=G$. For any subgroup $H \subseteq G$, $\Z[H]$ is a fusion subring of $\Z[G]$.
\end{Example}
\begin{Example} \label{eg:Fibfusionring}
    Let $\Fib \coloneqq \Z[x]/\< x^2=1+x\>$.
    Then $R$ is a fusion ring with $B = \{1,x\}$ and involution being the identity.
\end{Example}

For any fusion ring $R$ with basis $B=\{b_i\}_{i\in I}$, let $N_i$ as the matrix of left multiplication by $b_i$ in the basis $B$. Note that this is a non-negative matrix, which has a unique maximal eigenvalue $\lambda(N_i)$ that is necessarily a positive real number by the Frobenius--Perron theorem.
\begin{Definition}\label{def:FPdim}
Let $R$ be a fusion ring with basis $B$. 
For each $b_i \in B$, we define
\begin{equation}\label{eq:FP-dim-def}
    \FP(b_i) \coloneqq \lambda(N_i) \in \R_{>0},
\end{equation}
and we call it the \textbf{Frobenius--Perron dimension of} $b_i$.
The $\Z$-linear morphism $\FP:R\to \R$ defined on the basis elements as above is a ring homomorphism, and it is the unique ring homomorphism satisfying $\FP(b_i)>0$ for all $b_i \in B$ \cite[Proposition 3.3.6]{MR3242743}.
\end{Definition}
\iffalse
In fact, this assignment can be extended to all of $R$:
\begin{Proposition}[See \protect{\cite[Proposition 3.3.6]{MR3242743}}]\label{prop:fpdim} 
Let $\FP:R\to \R$ be the $\Z$-linear morphism defined on the basis elements as in \cref{eq:FP-dim-def}. Then,
    \begin{enumerate}
        \item $\FP:R\to \R$ is a ring homomorphism.
        \item \label{item:fpdim-b}There exists a unique (up to scaling) non-zero element $r\in R\otimes_\Z \R$ such that $xr=\FP(x)r$ for all $x\in R$, and it satisfies $ry=\FP(y)r$ for all $y\in R$. Furthermore, $r$ can be normalised to have positive coefficients, in which case $\FP(r)>0$.
        \item $\FP$ is the unique character (ring homomorphism $R\to \R$) of $R$ taking non-negative values on $B$, and $\FP(b_i)>0$ for all $b_i \in B$.
        \item If $x\in R$ has non-negative coefficients, then $\FP(x)=\lambda(N_x)$, where $N_x$ is the matrix of left (or equivalently right) multiplication by $x$.
    \end{enumerate}
\end{Proposition}
\begin{Definition}
    An element $r\in R\otimes_\Z \R$ as in \cref{prop:fpdim} \ref{item:fpdim-b} is called a \textbf{regular element} of $R$.
\end{Definition}
\begin{Proposition}[\protect{\cite[Proposition 3.3.9]{MR3242743}}]\label{prop:fpdim-*-invariant}
    If $R$ is a fusion ring with basis $B$, then $\FP$ is invariant under a bijection $*:B\to B$ that extends to an anti-automorphism of $R$, i.e 
    \[\FP(b_i^*)=\FP(b_i)\]
\end{Proposition}
\fi
For the rest of this subsection, $R$ denotes a fusion ring with basis $B$.
\begin{Definition}
    A \textbf{$\Z_+$-module over $R$} is a (left) $R$-module $M$ that is free as a $\Z$-module, with a fixed $\Z$-basis $\{m_\ell\}_{\ell\in L}$ satisfying
    \[b_i \cdot m_\ell=\sum_{k\in L}a_{i\ell}^km_k, \quad a_{i\ell}^k \in \Z_+.\]
    A \textbf{direct sum} of $\Z_+$-modules is the $\Z_+$-module with basis given by the  union of the bases of the summands.
    A $\Z_+$-module is decomposable if it is a direct sum of two non-trivial $\Z_+$-modules, otherwise it is \textbf{indecomposable}.\\
    A \textbf{$\Z_+$-submodule} is the $\Z$-span of a subset of the basis $\{m_{\ell}\}_{\ell \in L}$ that is closed under the \hbox{$R$-action}, i.e.\ there exist $K \subseteq L$ such that $R\cdot \Span_\Z\{m_k\}_{k \in K} =\Span_\Z\{m_k\}_{k\in K}$.
    A $\Z_+$-module is \textbf{irreducible} if it has no proper non-zero $\Z_+$-submodules, i.e.\ for all $\emptyset \neq K\subsetneq L$, $\Span_\Z\{m_k\}_{k\in K}$ is not a $\Z_+$-submodule. 
\end{Definition}
Note that for a $\Z_+$-module over a fusion ring, being irreducible and indecomposable are equivalent; see e.g.\ \cite[Exercise 3.4.3]{MR3242743}.
\begin{Example}[Regular module]
    Every fusion ring $R$ is an irreducible $\Z_+$-module over itself with the same basis $B$.
\end{Example}
\begin{Example}
    For a finite group $G$ and a (not necessarily normal) subgroup $H$ of $G$, $\Z[G/H]$ given by the formal $\Z$-span of (left) cosets of $H$ is a $\Z_+$-module for the fusion ring $\Z[G]$ as in \cref{eg:ZG-fusion-ring}, where $g\in \Z[G]$ acts on $g'H$ by $g\cdot g'H=(gg')H$. 
\end{Example}
\begin{Proposition}[\protect{\cite[Proposition 3.4.4]{MR3242743}}]\label{prop:reg-elt-module}
    Let $R$ be a fusion ring and let $M$ be an irreducible $\Z_+$-module over $R$ with basis $\{m_\ell\}_{\ell \in L}$. For each $x\in R$, let $[x]$ denote the matrix associated to the action of $x$ on $M$. Then up to scalar multiplication, there exists a unique $m\in M\otimes_\Z \R$ such that $m$ is a common eigenvector for all $[x]$, each with corresponding eigenvalue $\FP(x)$. Furthermore, $m$ can be rescaled so that it has positive coefficients.
\end{Proposition}
\begin{Definition}
    We call an element $m\in M\otimes_\Z \R$ in \cref{prop:reg-elt-module} with strictly positive coefficients a \textbf{regular element of $M$}.
\end{Definition}
Note that regular elements for $\Z_+$-modules is only defined up to (positive) scalar multiplication (though there is a canonical one for $R$ viewed as a $\Z_+$-module over itself; see \cite[Proposition 3.3.11]{MR3242743}).
%If we view $R$ as a $\Z_+$-module over itself, however, there is a preferred choice:
%\begin{Definition}
%    The element $r \coloneqq \sum_{i\in I}\FP(b_i)b_i \in R$ is a regular element (see \cite[Proposition 3.3.11]{MR3242743}), which we call the \textbf{canonical regular element} of $R$.
%\end{Definition}

\subsection{Hypergroups and fusion rings} \label{chap:HG}
\begin{Definition}[\protect{\cite[See Definition 2.4]{MR4886199}}] \label{def:HG}
    Let $\R[K]$ be  a (unital, associative) $\R$-algebra that is free as an $\R$-module with finite basis $K=\{r_i\}_{i\in I}$.
    Let $N_{ij}^k$ denote the structure constants defined by
    \[r_i\cdot r_j=\sum_{k\in I}N_{ij}^kr_k.\]
    We call $\R[K]$ a (finite) \textbf{hypergroup} if:
    \begin{itemize}
        \item $r_0$ is the unit of $\R[K]$ for some $r_0\in K$,
        \item $N_{ij}^k\geq 0$ for all $i,j,k\in I$,
        \item $\sum_{k\in I}N_{ij}^k=1$ for all $i,j\in I$, and
    \end{itemize}
    there exists an involution  $i\mapsto i^*$ of $I$ that extends to an anti-automorphism on $\R[K]$ such that $N_{ij}^0=N_{ji}^0$, and $N_{ij}^0> 0$ if and only if $j=i^*$. 
\end{Definition}
\begin{Remark}
    We define hypergroups to have real coefficients, however they are occasionally defined with complex coefficients, such as in \cite{MR4886199}.
\end{Remark}
\begin{Example}\label{eg:hypergroup-group}
    For every finite group $G$, its group algebra $\R [G]$ is a hypergroup with basis $G$ and involution $g\mapsto g^{-1}$.
\end{Example}
\begin{Example}
    The ring 
    \[\Fib_\R=\R[x]/(x^2=1+x)\]
    is a hypergroup with basis $\{1,\frac x\phi\}$, where $\phi=2\cos \frac \pi 5$, and involution given by the identity. Clearly multiplication by $1$ satisfies the structure requirements, and
    \[\frac x\phi \frac x\phi=\frac 1 {\phi^2}1+\frac 1 \phi \frac x\phi\]
    where $\frac 1 {\phi^2}+\frac 1 \phi=1$.
\end{Example}

The $\Z$-part of each of the two examples above is a fusion ring (cf.\ \cref{eg:ZG-fusion-ring} and \cref{eg:Fibfusionring} respectively) -- in fact, every fusion ring induces a hypergroup via extension of scalars to $\R$.
The following proof of this statement is not new; we only include it for completeness.
\begin{Proposition}[See \protect{\cite[Page 11]{MR4886199}}]\label{prop:Fusion-rings-are-HG}
    Every fusion ring $R$ with basis $\{b_i\}_{i\in I}$ and involution $(-)^*$ induces a hypergroup $R\otimes_\Z\R$ with basis $K \coloneqq\left\{\frac{b_i}{\FP(b_i)}\right\}_{i\in I}$ and involution $\left( \frac{b_i}{\FP(b_i)}\right)^*\coloneqq \frac{b_i^*}{\FP(b_i)}$
\end{Proposition}
\begin{proof}
    It is clear that $R\otimes_\Z\R$ is a unital associative $\R$-algebra with finite basis $K$. Furthermore since $1\in \{b_i\}_{i\in I}$ and $\FP(1)=1$ we get $1\in K$, so the unit is a basis element of $R\otimes_\Z \R$. Also, $(-)^*$ already induces an involution on $I$, so since $\FP(b_i^*)=\FP(b_i)$ \cite[Proposition 3.3.9]{MR3242743}, it will be an involution on $K$ as defined. By linearity, it will extend to an anti-automorphism on $R\otimes_\Z \R$, as $^*$ extends to an anti-automorphism on $R$. Furthermore, 
    \begin{align*}\frac{b_i}{\FP(b_i)}\frac{b_j}{\FP(b_j)}=\sum_{k\in I}\frac{c_{ij}^kb_k}{\FP(b_i)\FP(b_j)}
    =\sum_{k\in I}\frac{\FP(b_k)c_{ij}^k}{\FP(b_i)\FP(b_j)}\frac{b_k}{\FP(b_k)}\end{align*}
    So since $\FP(b_i)> 0$ for all $i\in I$, we have 
    \[N_{ij}^k:=\frac{\FP(b_k)c_{ij}^k}{\FP(b_i)\FP(b_j)}\geq 0\]
    Since $\FP: R \to \R$ is a ring homomorphism, we have
    \begin{align*}
    \FP(b_i)\FP(b_j)=\FP(b_ib_j)&=\FP\left(\sum_{k\in I} c_{ij}^k b_k\right) =\sum_{k\in I}c_{ij}^k\FP(b_k),
    \end{align*}
    which gives us
    \[
    \sum_{k\in I}N_{ij}^k=\frac{1}{\FP(b_i)\FP(b_j)}\sum_{k\in I}\FP(b_k)c_{ij}^k=1.
    \]
    We also have,  
    \begin{align*}N_{ij}^0&=\frac{c_{ij}^0}{\FP(b_i)\FP(b_j)}=\frac{\delta_{ij^*}}{\FP(b_i)\FP(b_j)}\\
    &=\frac{\delta_{i^*j}}{\FP(b_j)\FP(b_i)}=\frac{c_{ji}^0}{\FP(b_j)\FP(b_i)}=N_{ji}^0\end{align*}
    and $N_{ij}^0>0\iff \delta_{ij^*}\neq0\iff i=j^*$. As such, $R\otimes_\Z\R$ is a hypergroup.
\end{proof}
%\begin{Example}
%    Given a finite group $G$, the fusion ring $\Z[G]$ from \cref{eg:ZG-fusion-ring} induces the hypergroup $\R[G] = \Z[G]\otimes_\Z\R$ given in \cref{eg:hypergroup-group}.
%\end{Example}

\subsection{Action of hypergroups  and \texorpdfstring{$\Z_+$}{non-negative Z}-modules}\begin{Definition}\label{def:HG-action}
    Let $\R[K]$ be a hypergroup. We define a \textbf{hypergroup action} on $V\cong\R^n$ as a map $\Theta:\R[K]\to \Mat_n(\R)$ such that $\Theta$ is a $\R$-algebra homomorphism.
\end{Definition}
\begin{Remark}
    It is clear that we are using very little structure from the hypergroup (e.g.\ the involution is not involved). This will, however, be sufficient for our purposes.
\end{Remark}
\begin{Example}
    For the hypergroup $\R[G]$ as in \cref{eg:hypergroup-group}, an $\R[G]$-action on $V$ in the sense of \cref{def:HG-action} is just a real representation of the group algebra, which is equivalently a $G$-representation.
\end{Example}

%\begin{Example}
%    Every hypergroup $\R[K]$ acts on itself.
%\end{Example}
\begin{Definition}
    A \textbf{fixed point} of a hypergroup action on $V$ is a point $x\in V$ such that $\Theta(r_i)x = x$ for all $i\in I$.
\end{Definition}
Note that the set of fixed points is always a linear subspace of $V$.
\begin{Example}
    For a hypergroup $\R[G]$ as in \cref{eg:hypergroup-group}, the fixed points of the $\R[G]$-action on $V$ are exactly the fixed points of the $G$-action, i.e.\ $x\in V$ such that $g\cdot x=x$ for all $g\in G$.
\end{Example}

Recall from \cref{prop:Fusion-rings-are-HG} that given a fusion ring $R$ with basis $\{b_i\}_{i\in I}$, we get a hypergroup $R\otimes_\Z\R$ with basis $\left\{ \frac{b_i}{\FP(b_i)}\right\}_{i\in I}$.
The following proposition shows that for hypergroup actions induced from their $\Z_+$-modules, the subspace of fixed points is exactly the span of regular elements of the irreducible components.
\begin{Proposition}\label{prop:Z-plus-action-hg}
    Let $R$ be a fusion ring.
    Then any $\Z_+$-module $M$ over $R$ extends to a hypergroup action of $R\otimes _\Z \R$ on $V \coloneqq M\otimes_\Z \R$. Furthermore, if $M = \bigoplus_{k=1}^m M_k$ is the decomposition of $M$ into irreducible $\Z_+$-modules, then the subspace of fixed points of the hypergroup action on $V$ is the $\R$-span of regular elements $r_k$ of $M_k$.
\end{Proposition}
\begin{proof}
    The first statement follows immediately from extensions of scalars.
    
    Now choose a regular element $r_k \in M_k\otimes_\Z \R$ for each irreducible $\Z_+$-module $M_k$.
    We will prove that the subspace of fixed points $\mathfrak{F}$ of the hypergroup $R\otimes _\Z \R$ acting on $M\otimes_\Z\R$ is exactly $\Span_\R\{r_k\}_{k=1}^m$.
    Throughout this proof $B=\{b_i\}_{i\in I}$ denotes the basis of $R$.

    The containment $\Span_\R\{r_k\}_{k=1}^m \subseteq \mathfrak{F}$ is immediate. Indeed, being a regular element of $M_k$ gives $b_i\cdot r_k = \FP(b_i)r_k$ for all $i \in I$, which is equivalent to $\frac{b_i}{\FP(b_i)}\cdot r_k = r_k$. Thus $r_k \in \mathfrak{F}$ and the containment follows from the linearity of the subspace $\mathfrak{F}$.
    
    Conversely, let $v \in \mathfrak{F}$, so that $b_i\cdot v =\FP(b_i)v$ for all $i \in I$.
    The decomposition $M = \bigoplus_{k=1}^m M_k$ gives us a direct decomposition
    \[
    M\otimes_\Z \R = \bigoplus_{k=1}^m M_k\otimes_\Z \R,
    \]
    so we can uniquely express $v = \sum_{k=1}^m v_k$ with each $v_k \in M_k\otimes_\Z \R$.
    As such, we obtain $\sum_{k=1}^m b_i\cdot v_k  = \sum_{k=1}^m\FP(b_i)v_k$.
    Since $b_i\cdot v_k \in M_k\otimes_\Z R$, the direct decomposition allows us to further deduce that $b_i \cdot v_k = \FP(b_i)v_k$ for all $i \in I$ and all $1\leq k \leq m$.
    It follows that $x\cdot v_k = \FP(x)v_k$ for all $x \in R$.
    By assumption, each $M_k$ is an irreducible $\Z_+$-module over $\R$, so by \cref{prop:reg-elt-module}, $v_k$ is a scalar multiple of $r_k$.
    It follows that $v = \sum_{k=1}^m v_k \in \Span\{r_k\}_{k=1}^m$, as required.
\end{proof}

\section{Verlinde fusion ring and its \texorpdfstring{$\Z_+$}{Z}-modules} \label{sec:verlindeclassification}
In this section we define the Verlinde fusion ring $R_n$ and recall the ADE(T) classification of its $\Z_+$-modules.
\subsection{Chebyshev polynomials}
\begin{Definition}
    Let the family of polynomials $\{\Delta_i(x)\}_{i\in \N}\subset \Z[x]$ be defined by the following recursive relation: 
    \begin{equation}\label{eq:recursion-def}
    \begin{split}
        \Delta_0(x)&\coloneqq 1, \qquad \Delta_1(x)\coloneqq x \\
        \Delta_k(x)&:=x\Delta_{k-1}(x)-\Delta_{k-2}(x),\qquad\text{for all } k\geq 2
    \end{split}
    \end{equation}
    We often suppress the `$x$' for  cleaner notation, by writing $\Delta_i:=\Delta_i(x)$. 
    These are known as the (normalised) \textbf{Chebyshev polynomials} of the second kind.
\end{Definition}
\begin{Example}\label{eg:first-chebychevs}
    The next few polynomials are
    \begin{align*}
        \Delta_2(x)&=x^2-1, \qquad &&\Delta_3(x)=x^3-2x\\
        \Delta_4(x)&=x^4-3x^2+1, \qquad &&\Delta_5(x)=x^5-4x^3+3x.
        %\Delta_6(x)&=x^6-5x^4+6x^2-1
    \end{align*}
\end{Example}
\begin{Proposition}\label{prop:odd-even-delta}
    If $k$ is even, $\Delta_k(x)$ contains only even powers of $x$ and if $k$ is odd, $\Delta_k(x)$ contains only odd powers of $x$.
\end{Proposition}
\begin{proof}
    The claim clearly holds for $k=0,1$. If $k$ is even, then $k-1$ is odd and $k-2$ is even, so if $\Delta_{k-1}(x)$ consists only of odd powers of $x$ and $\Delta_{k-2}(x)$ consists of only even powers, then $\Delta_k(x)=x\Delta_{k-1}(x)-\Delta_{k-2}(x)$ will only consist of even powers of $x$. A similar statement holds when $k$ is odd, and thus the proposition holds by induction.
\end{proof}
A direct calculation using induction shows that for all $i,j\geq 0$, we have
    \begin{equation}\label{eq:multiplication}
        \Delta_i\Delta_j=\Delta_{|i-j|}+\Delta_{|i-j|+2}+\dots+\Delta_{i+j}
        \end{equation}
Moreover, the following trigonometric identity holds for all $k\geq 0$ and all $y\in \R$:
\begin{equation}\label{eqn:chebychev-trig}
    \Delta_k(2\cos y)=\frac{\sin (k+1)y}{\sin y}.
\end{equation}

%%%%%%%%%%%%%%%%%%%%%%%%%%%%%%%%%%%%%%%%%%%%%%%%%%%%%%
% Hidden calculation proof
%%%%%%%%%%%%%%%%%%%%%%%%%%%%%%%%%%%%%%%%%%%%%%%%%%%%%%
\iffalse
\begin{proof}
    For $k=0$, $\Delta_0(2\cos x)=1=\frac{\sin x}{\sin x}$. For $k=1$,
    \[\Delta_1(2\cos x)=2\cos x = \frac{\sin 2x}{\sin x}\]
    For all $k\geq 2$, if the proposition holds for $k-1$ and $k-2$, then by \cref{eq:recursion-def},
    \begin{align*}\Delta_{k}(2\cos x)&=2\cos x\Delta_{k-1}(2\cos x)-\Delta_{k-2}(2\cos x)\\
    &=2\cos x \frac{\sin kx}{\sin x} - \frac{\sin (k-1)x}{\sin x}\\
    &=\frac{1}{\sin x}(2\cos x\sin kx-\sin (k-1)x)\\
    &= \frac{\sin (k+1)x}{\sin x}
    \end{align*}
    The result follows by induction.
\end{proof}
\fi

\subsection{The Fusion Ring $R_n$}
\begin{Lemma}\label{lem:negatives}
    If $\Delta_n=0$, then for all $0\leq k <n$ we have
    \[\Delta_{n+k}=-\Delta_{n-k}\]
\end{Lemma}
\begin{proof}
    Since $\Delta_n=0$, the $k=0$ case is trivial and the $k=1$ case follows from $\Delta_{n+1} = \Delta_1\Delta_n-\Delta_{n-1}=-\Delta_{n-1}$.
    Now suppose that the statement holds for $\Delta_{n+k}$ and $\Delta_{n+(k-1)}$, then \begin{align*}-\Delta_{n-k-1}-\Delta_{n-k+1}&=-\Delta_1\Delta_{n-k}=\Delta_1\Delta_{n+k}\\
    &=\Delta_{n+k+1}+\Delta_{n+k-1}=\Delta_{n+k+1}-\Delta_{n-k+1}.
    \end{align*}
    It follows that $\Delta_{n+k+1}=-\Delta_{n-k-1}$ and the result holds by induction.
\end{proof}
\begin{Definition}
    For all $n\geq 0$, we define the ring
    \[R_n:=\Z[x]/(\Delta_n(x)).\]
    It is a (commutative) fusion ring with $\Z$-basis $\mathfrak{B}=\{\Delta_k\}_{0\leq k \leq n-1}$ and involution of the basis elements $(-)^*$ given by the identity. Using \cref{lem:negatives} and \cref{eq:multiplication}, the multiplication of the basis elements is given by
    \begin{equation}\label{eq:multiplication-2}
    \Delta_i\Delta_j= \begin{cases}
        \Delta_{|i-j|}+\Delta_{|i-j|+2}+\dots+\Delta_{i+j}, \text{ if } i+j\leq n \\
        \Delta_{|i-j|}+\Delta_{|i-j|+2}+\dots+\Delta_{2n-(i+j)-2} , \text{ if } i+j> n.
        \end{cases}
    \end{equation}
    We call this the \textbf{Verlinde fusion ring} (of $\mathfrak{sl_2}$ at level $(n-1)$) \cite[Definition 3.1]{MR1333750}. 
    Moreover, we also define 
    \[
    R_n^\e \coloneqq \Span_\Z\left\{\Delta_{2k}\,|
    \,0\leq k\leq \floor*{\frac{n-1}{2}}\right\}\subseteq R_n
    \]
    to be the subring generated by the even-labelled basis elements, which we call the \textbf{even part} of the Verlinde fusion ring.
    It follows immediately from the multiplication rule \cref{eq:multiplication-2} that $R_n^\e$ is a fusion subring of $R_n$.
\end{Definition}
Using \cref{eqn:chebychev-trig}, we see that the ring homomorphism $R_n \to \R$ defined by sending $x \mapsto \Delta_1(2 \cos(\pi/(n+1)))$ is well-defined, and moreover $\Delta_i(2 \cos(\pi/(n+1)) > 0$ for all $0 \leq i < n$.
The uniqueness property of $\FP$ (see \cref{def:FPdim}) implies that this completely determines $\FP: R_n \to \R$, so that
\begin{equation} \label{eq:FPdim-verlinde}
    \FP(\Delta_i) = \Delta_i(2 \cos(\pi/(n+1))
\end{equation}
for all $0 \leq i \leq n-1$.
\begin{Proposition}\label{prop:R-even-delta-2}
    Any element in $R_n^\e$ can be expressed as a polynomial in $\Delta_2$, with coefficients in $\Z$.
    In particular, $R_n^\e$ is generated by $\Delta_2$ as a ring.
\end{Proposition}
\begin{proof}
    By \cref{prop:odd-even-delta}, for all $0\leq k\leq \floor{\frac{n-1}{2}}$, the unique coset representative of $\Delta_{2k}\in R_n^\e$ with degree less than $n$ is a polynomial over $\Z$ with only even powers of $x$, i.e.\ it is a polynomial in $x^2$ with coefficients in $\Z$. Since $x^2=\Delta_2+1$, we get that $\Delta_{2k}$ is indeed a polynomial in $\Delta_2$ with coefficients in $\Z$. 
    The result follows from the fact that the elements $\Delta_{2k}$ form a basis for $R_n^\e$.
    %As the result holds on the $\Z_+$-basis of $R_n^\e$, it extends to all of $R_n^\e$ by $\Z$-linearity. 
\end{proof}
\subsection{Irreducible \texorpdfstring{$\Z_+$}{non-negative Z}-Modules of \texorpdfstring{$R_{h-1}$}{Verlinde fusion rings}.}
We now recall the classification of $\Z_+$-modules over the Verlinde fusion ring $R_{h-1}$.
\begin{Definition} \label{def:VerlindeZ+module}
    Let $\Gamma$ be an ADE Dynkin diagram (see \cref{fig:CoxDyn}) whose corresponding Coxeter system is rank $n$ and has Coxeter number $h$.
    Since $\Gamma$ is a graph (without labels), and we let $Ad_\Gamma$ denote its adjacency matrix.
    Let 
    \[
    M_\Gamma \coloneqq \Span_\Z\{\alpha_1,\ldots,\alpha_n\} \subset V_\Gamma
    \]
    be the lattice in $V_\Gamma \coloneqq \Span_\R\{\alpha_1,\ldots,\alpha_n\}$.
    We define $M_\Gamma$ to be the $\Z_+$-module over $R_{h-1}$ given by the ring homomorphism $R_{h-1} \to \End_\Z(M_\Gamma)$ defined by $\Delta_1 \mapsto Ad_\Gamma = 2I_n-\mathfrak{A}$, where $\mathfrak{A}$ is the Cartan matrix from \cref{def:symmetricCartan}.
\end{Definition}
The fact that each of these defines an irreducible $\Z_+$-module is given in \cite{MR1333750,MR1976459}.
Moreover, these are exactly the $\Z_+$-modules which are Grothendieck groups of module categories over fusion categories associated to $U_q(\mathfrak{sl}_2)$ (whose fusion rings are the Verlinde fusion rings).
\begin{Remark}
    We note that there is a slight calculation error in the proof for the type D case in \cite[Theorem 3.4]{MR1333750} (eq.\ (3.2) is wrong), though the fact that these arise as Grothendieck groups of module categories bypasses this issue. 
    A fix of this error can also be found in \cite[Chapter 4]{Boyle_hons}.
    We also note that there is another ``tadpole'' series T which is part of the classification of $\Z_+$-modules over $R_{h-1}$, though these do not arise as Grothendieck groups of module categories; see \cite{MR1976459} for more details.
\end{Remark}

\section{Main theorem} \label{sec:mainthm}
Henceforth, we fix $\Gamma$ to be an ADE Dynkin diagram with its set of vertices $\Gamma_0 = \indexset = \{1,2,\ldots, n\}$ ordered as in \cref{fig:CoxDyn}, so that the bifurcation of $\Gamma_0 = \Gamma_0^+ \sqcup \Gamma_0^-$ defined by \cref{eq:bifurcation} makes $\Gamma$ a bipartite graph.
We let $(W,S)$ be the corresponding irreducible Coxeter system -- so the Coxeter group $W$ is finite and has rank $n$. The Coxeter number of $(W,S)$ will be denoted by $h$.
Let $V\coloneqq V_\Gamma = \Span_\R\{\alpha_1,\ldots,\alpha_n\}$ be the standard geometric representation of $W$, with the Coxeter plane siting as a subspace of $V$ (\cref{def:Coxeterplane}).

Let $R_{h-1}$ be the Verlinde fusion ring and recall from \cref{def:VerlindeZ+module} the $\Z_+$-module $M:=M_\Gamma$ over $R_{h-1}$, which we view as a lattice $M=\Span_\Z\{\alpha_1,\ldots,\alpha_n\}$ sitting naturally sitting inside $V$.
By restricting to the even part fusion subring $R_{h-1}^\e \subseteq R_{h-1}$, $M$ is a $\Z_+$-module over $R_{h-1}^\e$ and we obtain an action of the hypergroup $R_{h-1}^\e\otimes_\Z\R$ on $M\otimes_\Z\R = V$ (see \cref{prop:Z-plus-action-hg}).
Our main theorem is as follows.
\begin{Theorem}\label{thm:main-theorem}
    The subspace of fixed points of the hypergroup $R_{h-1}^\e \otimes_\Z \R$ acting on $V =M\otimes_\Z \R$ is the Coxeter plane.
\end{Theorem}

The rest of this section is dedicated to the proof of the theorem.
\begin{Lemma} \label{lem:bifurcationofbasis}
    Let $X^+\sqcup X^-$ be the bipartition of the $\Z$-basis of $M$ induced by the bifurcation of $\Gamma_0$, so that $X^+ = \{ \alpha_i \mid i \in \Gamma_0^+ \}$ and $X^- = \{ \alpha_i \mid i \in \Gamma_0^- \}$. Then $\Delta_1 \cdot(\Span_\Z X^\pm)\subseteq \Span_\Z X^\mp$.
\end{Lemma}
\begin{proof}
    The action of $\Delta_1$ on $M$ is defined by the matrix $Ad_\Gamma = 2I_n-\mathfrak{A}$, so it suffices to show that $2I_n-\mathfrak{A}$ is block anti-diagonal with respect to the basis partition $X^+ \sqcup X^-$.
    Let $x\neq x'\in X^+$ and $y\neq y'\in X^-$.
    By the bipartite property, we get that $\langle \,x\mid x'\,\rangle=0$ and $\langle \,y\mid y'\,\rangle=0$. Furthermore, $\langle\,x\mid x\,\rangle=\langle\,y\mid y\,\rangle=-2$, so the entries of $2I_n-\mathfrak{A}$ corresponding to $(x,x')$ and $(y,y')$ are zero for all $x,x'\in X^+$ and $y,y'\in X^-$. This shows that $2I_n-\mathfrak{A}$ is block anti-diagonal, as required.
\end{proof}
\begin{Lemma}\label{lem:restrictiondecomposition}
    View $M$ as a $\Z_+$-module over $R_{h-1}^\e \subseteq R_{h-1}$.
    Then $M$ decomposes into a direct sum of two irreducible $\Z_+$-modules over $R_{h-1}^\e$
    \begin{equation}\label{eq:direct-sum}M=M^+\oplus M^-,\end{equation}
    with $M^{\pm}$ induced by the bifurcation in \cref{lem:bifurcationofbasis}, i.e.\ $M^{\pm} = \Span_\Z X^\pm$.
\end{Lemma}
\begin{proof}
    Any element in $R_{h-1}^\e$ can be expressed as a polynomial in $\Delta_2 \in R_{h-1}$ with coefficients in $\Z$ by \cref{prop:R-even-delta-2},
    so to show that $M^\pm$ are $R_{h-1}^\e$-modules, we only need to show closure under $\Delta_2$.
    We already have that $\Delta_1\cdot M^\pm \subseteq M^\mp$ by the lemma before, so $\Delta_1^2\cdot M^\pm\subseteq M^\pm$. Sinc $\Delta_0=1$ acts trivially, we have $\Delta_2\cdot M^\pm=(\Delta_1^2-\Delta_0)\cdot M^\pm\subseteq M^\pm$ and so $M^\pm$ are $\Z_+$-modules over $R_{h-1}^\e$.
    
    We now prove irreducibility. For all $x\neq x'\in X^+$, there is a path of even length $2k$ from $x$ to $x'$ on the Coxeter graph by construction of the bifurcation (see \cref{eq:bifurcation}). Thus, as $\Delta_1$ acts on $M$ via the adjacency matrix of the Coxeter graph, the $(x,x')$-entry of the matrix corresponding to the action of $\Delta_1^{2k}=(\Delta_2+\Delta_0)^k\in R_{h-1}^\e$ encodes exactly the number of walks from $x$ to $x'$ on the graph and is thus non-zero, so any non-zero $\Z_+$-submodule of $M^+$ must contain the entire module $M^+$. The same argument applies for $M^-$, replacing $X^+$ with $X^-$.  
\end{proof}

\begin{Lemma}\label{lem:reg-elt-decomp}
    Fix $r\in M \otimes _\Z\R $ to be a regular element of $M$ with $M$ viewed as an irreducible $\Z_+$-module over $R_{h-1}$. Let $r=r^++r^-$ be the decomposition under the direct sum
    \[M\otimes _\Z\R =(M^+\otimes _\Z\R )\oplus (M^-\otimes _\Z\R )\]
    induced by the decomposition in \cref{lem:restrictiondecomposition} via extending scalars to $\R$.
    Then $r^\pm$ are regular elements of the irreducible $\Z_+$-modules $M^\pm$ over $R_{h-1}^\e$ respectively, and 
    \begin{equation} \label{eq:delta1onsumanddifference}
        \begin{split}
            \Delta_1\cdot (r^++r^-)&=\FP(\Delta_1)(r^++r^-) \\
            \Delta_1\cdot (r^+-r^-)&=-\FP(\Delta_1)(r^+-r^-)
        \end{split}
    \end{equation}
\end{Lemma}
\begin{proof}
    For all $a\in R_{h-1}^\e$,  we have
    \[
    a\cdot r^+ + a\cdot r^-= a\cdot r = \FP(a)r = \FP(a)r^++\FP(a)r^-.
    \]
    By the direct sum decomposition $M=M^+\oplus M^-$ and the closure of $M^\pm$ under $R_{h-1}^\e$-action, we further deduce that $a\cdot r^\pm=\FP(a)r^\pm$. Note that since $r$ has positive coefficients in the basis $X^+\sqcup X^- = \{\alpha_i\}_{i=1}^n$, $r^\pm$ will have positive coefficients in $X^\pm$ as well, so
    that $r^\pm$ is a regular element of $M^\pm$ respectively. 

    The first equation in \cref{eq:delta1onsumanddifference} is immediate from $r = r^+ + r^-$ being a regular element of $M$ as a $\Z_+$-module over $R_{h-1}$.
    For the final equation, since $\Delta_1\cdot M^\pm\subseteq M^\mp$ by \cref{lem:bifurcationofbasis} and we have the direct sum decomposition $M=M^+\oplus M^-$, the following equation
    \[
    \Delta_1\cdot r^+ + \Delta_1 \cdot r^-= \Delta_1\cdot r = \FP(\Delta_1)r=\FP(\Delta_1)r^- + \FP(\Delta_1)r^+
    \]
    further implies
    \[\Delta_1\cdot r^+=\FP(\Delta_1)r^-\quad\text{and}\quad \Delta_1\cdot r^-=\FP(\Delta_1)r^+.\]
    This combines to give the final desired equation:
    \[\Delta_1\cdot (r^+-r^-)=\FP(\Delta_1) r^--\FP(\Delta_1)r^+=-\FP(\Delta_1)(r^+-r^-).\]
\end{proof}
Now we have all the tools to prove our main theorem.

\begin{proof}[Proof of \cref{thm:main-theorem}]
    Choose $r \in M$ to be a regular element of $M$ with $M$ viewed as an irreducible $\Z_+$-module over $R_{h-1}$.
    Now view $M$ as a $\Z_+$-module over the fusion subring $R_{h-1}^\e$, so that $M$ decomposes into $M^+\oplus M^-$ with each $M^\pm$ irreducible over $R_{h-1}^\e$.
    Following \cref{lem:reg-elt-decomp}, let $r=r^++r^-$ be the corresponding decomposition obtained from 
    \[
    M\otimes_\Z \R = (M^+\otimes_\Z \R) \oplus (M^-\otimes_\Z \R),
    \]
    with $r^\pm$ regular elements of $M^\pm$ over $R_{h-1}^\e$ respectively.
    By \cref{prop:Z-plus-action-hg}, the subspace of fixed points of the hypergroup $R_{h-1}^\e\otimes_\Z\R$ acting on $V = M\otimes_\Z\R$ is given by $\Span_\R\{r^+,r^-\}$.

    Clearly, $\Span_\R\{r^+,r^-\} = \Span_\R\{r^++r^-,r^+-r^-\}$.
    We shall prove that this is the Coxeter plane by showing that there exists $c,c' \in \R$ such that $r^+ + r^- = cu^+$ and $r^+-r^- = c'u^-$, with $u^\pm$ as given in \cref{thm:Coxeter-plane-span}, which are the elements spanning the Coxeter plane (\cref{def:Coxeterplane}).
    Recall that by definition, that the matrix of $\Delta_1$ acting on $M$ is given by $2I_n-\mathfrak A$. Moreover, by \cref{eq:FPdim-verlinde}, we have $\FP(\Delta_1)=2\cos(\frac{\pi}{h})$. By \cref{lem:reg-elt-decomp}, we get that $r^+ \pm r^-$ are $\pm2\cos(\frac{\pi}{h})$-eigenvectors respectively for $2I_n-\mathfrak A$.
    Since the eigenvalues $\pm2\cos(\frac{\pi}{h})$ of $2I_n-\mathfrak A$ are simple by \cref{thm:Coxeter-plane-span}, we get that $r^+ \pm r^-$ is a scalar multiple of $u^\pm$ respectively, as required.
\end{proof}

\bibliographystyle{plain} % We choose the "plain" reference style
\bibliography{bib} % Entries are in the refs.bib file

\end{document}